\documentclass[twoside, a4paper, 12pt]{article}

\usepackage[top=3.0cm, bottom=3.0cm, left=3.6cm, right=3.6cm]{geometry}

\usepackage{amsmath, amssymb, amsthm}
\usepackage[mathscr]{eucal}
\usepackage[dvips]{graphicx}
\usepackage{float}
\usepackage{bbold}
\usepackage{cancel}
\usepackage{amsfonts}
\usepackage{dsfont}

\usepackage{color}

\usepackage{enumitem}

\usepackage[colorlinks=blue,linkcolor=blue]{hyperref}

\newcommand{\R}{{\rm I\!R}}

\newtheorem{theorem}{Theorem}[section]
\newtheorem{lemma}[theorem]{Lemma}
\newtheorem{definition}[theorem]{Definition}
\newtheorem{defi}[theorem]{Definition}

\newtheorem{corollary}[theorem]{Corollary}
\newtheorem{proposition}[theorem]{Proposition}

\theoremstyle{remark}
\newtheorem{remark}[theorem]{Remark}

\newcommand{\mintG}{\displaystyle {\int \kern -0.961em -}_{\Gamma_t}}
\newcommand{\minto}{{\int \kern -0.961em -}_{\Omega}}

\DeclareMathOperator{\esssup}{ess \, \, sup}
\DeclareMathOperator{\essinf}{ess \, \, inf}

\def \e {\varepsilon}

\newcommand{\mintO}{\displaystyle {\int \kern -0.961em -}_{\Omega}}
\newcommand{\mintOs}{\displaystyle {\int \kern -0.961em -}_0^{|\Omega|}}

\newcommand{\tn}[1]{{\color{black}{#1}}}
\newcommand{\tnr}[1]{{\color{black}{#1}}}

\begin{document}

\title{ \tnr{On the $\omega$-limit set of a nonlocal differential equation: application of rearrangement theory}}

\author{Thanh Nam Nguyen\thanks{\texttt{Laboratoire de Math\'{e}matique, Analyse Num\'{e}rique et EDP, Universit\'e de Paris-Sud, F-91405 Orsay Cedex, France}}}
\date{\today}

\maketitle

\medskip

\noindent{\bf Abstract.}
We study the $\omega$-limit set of solutions of a nonlocal ordinary differential equation, where the nonlocal term is such that the space integral of the solution is conserved in time. 
Using the monotone rearrangement theory,
we show that the rearranged equation in one space dimension is the same as the original equation in higher space dimensions. In many cases,
 this property allows us to characterize the $\omega$-limit set for the nonlocal differential equation. More precisely, we prove that 
 the $\omega$-limit set 
only  contains one element.

\bigskip

\section{Introduction}

The aim of the present paper is to study the $\omega$-limit set of solutions of 
the initial value problem
\begin{equation*}\label{chap3:3chapchaporiginal:eq}
(P) \ \ \ \left\{ 
\begin{aligned}
&u_t =  g(u)p(u)-g(u) \,\frac{\displaystyle\int_\Omega g(u)p(u) }{\displaystyle\int_\Omega g(u)}   &    & \quad \quad x\in \Omega,\  t \ge 0,
\\
&u(x,0)=u_0(x) & & \quad \quad x\in \Omega.
\end{aligned}
\right.
\end{equation*}
Here $\Omega \subset \R^N (N \ge 1)$ is an open bounded set, 
$g, p: \R \to \R$  are continuously differentiable
 and $u_0$ is a bounded function. More precise conditions on $g, p$ and 
 $u_0$ will be given later.
A typical example is given by the functions  $g(u)=u(1-u)$ and $p(u)=u$. In this case, the equation becomes
$$u_t =  u^2(1-u)-u(1-u) \,\frac{\displaystyle\int_\Omega u^2(1-u)}{\displaystyle\int_\Omega u(1-u)}.$$
\tnr{The corresponding parabolic equation 
$$u_t =  \Delta u +\frac{1}{\e^2}\left( u^2(1-u)-u(1-u) \,\frac{\displaystyle\int_\Omega u^2(1-u)}{\displaystyle\int_\Omega u(1-u)} \right).$$
has been used by Brassel and Bretin\cite[Formula (9)]{BrasselBretin} to approximate mean curvature flow with volume conservation. It has been also proposed by Nagayama \cite{nagayama} to describe a bubble motion with a chemical reaction.
He supposes furthermore that the volume of the bubble is preserved in time.
Mathematically, it is expressed in the form of the mass conservation property}
\begin{equation}\label{eq:mass:conservation34}
	\int_\Omega u(x,t)\,dx=\int_\Omega u_0(x)\,dx \quad \mbox{for all}\ \ t \ge 0.
\end{equation}
 \tn{We refer to Proposition \ref{lem:local:existence} for a rigorous proof of this equality.}
\medskip 

We will consider Problem $(P)$ under some different hypotheses on the initial function $u_0$. 
Problem $(P)$ possesses a Lyapunov functional whose form depends on the hypothesis satisfied by 
$u_0$ (see section 4 for more details).

\medskip

In this paper, we always consider the following hypotheses on the functions $g$ and $p$:
\begin{equation*}
\left\{ 
\begin{aligned}
&p \in C^1(\R) \ \ \mbox{is strictly increasing on} \ \R,\vspace{8pt} \\ 
&g \in C^1(\R), g(0)=g(1)=0, g>0 \mbox{ on } (0, 1) \mbox{ and } g<0 \mbox{ on }(-\infty,0)\cup(1,\infty).
\end{aligned}
\right.
\end{equation*}
We suppose that the initial function satisfies one of the following hypotheses:

\noindent {$\bf (H_1)$} \,
$u_0 \in L^\infty(\Omega)$, $u_0(x) \ge 1 \mbox{~~for a.e.~~}x\in\Omega,  \mbox{~~and~~} u_0 \not\equiv 1.$
\vskip 0.10cm
\noindent {$\bf (H_2)$} \, 
$u_0 \in L^\infty(\Omega)$, $0 \le  u_0(x) \le 1 \mbox{~~for a.e.~~}x\in\Omega,  \mbox{~~and~~} \int_\Omega g(u_0(x))\,dx \neq 0.$
\vskip 0.10cm
\noindent {$\bf (H_3)$} \, 
$u_0 \in L^\infty(\Omega)$, $u_0(x) \le 0 \mbox{~~for a.e.~~}x\in\Omega,  \mbox{~~and~~} u_0 \not\equiv 0.$

\noindent Note that Hypothesis $\bf (H_1)$ (and also $\bf (H_3)$) implies that $\int_\Omega g(u_0) \neq 0$.

\medskip

Before defining a solution of Problem $(P)$, we introduce the notation 
\begin{equation}\label{eq:defi:F}
F (u):=g(u)p(u)-g(u)\dfrac{\displaystyle \int_\Omega g(u)p(u)}{\displaystyle \int_\Omega g(u)}.
\end{equation}
\begin{definition}\label{thedefinitionofsolution}
Let $0<T \le \infty$. The function $u\in C^1([0,T);L^\infty(\Omega))$ is called a solution of Problem $(P)$ on $[0,T)$ if  the three following properties hold
\begin{enumerate}[label=\emph{(\roman*)}]
\item $u(0)=u_0$,
\item $\displaystyle \int_\Omega g(u(t)) \neq 0 \mbox{~~for all~~} t\in[0, T)$,
\item
$\dfrac{du}{dt}=F(u) \mbox{~~\tnr{in the whole interval}~~}[0, T)$.
\end{enumerate}
\end{definition}

\medskip

\tn{The $\omega$-limit sets}
are
important and interesting 
objects in the theory of dynamical systems. 
Understanding their structure allows us to apprehend
the long time behavior of solutions of dynamical systems. 
In this paper, we characterize  the $\omega$-limit set of solutions of 
Problem $(P)$, which is defined as follows:
 \begin{defi}\label{defi:omega:limit}
We define the $\omega$-limit set of $u_0$ by
$$\omega(u_0):=\{\varphi \in L^1(\Omega): \exists t_n \to \infty, u(t_n) \to \varphi \mbox{~~in~~} L^1(\Omega)\mbox{~as~} n \to \infty\}.$$
\end{defi}

In the above definition, we do not use the $L^\infty$-topology
to define $\omega(u_0)$ because 
the solution often develops sharp transition layers which cannot
be captured by the $L^\infty$-topology.
\tn{Note also that} as we will see in Theorem \ref{thm:existence:boundedness}, solutions of $(P)$ are uniformly bounded so that the topology of $L^1$
 is equivalent to that of $L^p$ with $p \in [1, \infty)$. 
For convenience, we refer to the books \cite{robinson, temam} 
for studies about dynamical  systems as well as 
the structure of $\omega$-limit sets.

\medskip

An essential step to study $\omega(u_0)$ is to show 
the relative compactness of the solution orbits in $L^1(\Omega)$. 
In local problems, 
the standard comparison principle can be \tn{applied} to obtain 
the uniform boundedness of solutions. 
Furthermore, in local problems with a diffusion term, such as local parabolic problems, 
the uniform boundedness of solutions
 implies the relative compactness of solution orbits in some suitable spaces by using Sobolev imbedding theorems.
However, the above scheme cannot be applied to Problem $(P)$, due to the presence of the nonlocal term as well as to the lack of a diffusion term.

\medskip

By careful observation of the dynamics of pathwise trajectories 
(i.e. the sets $ \{ u(x,t): t \ge 0 \} $ for $x \in \Omega$), we show the existence of invariant sets and
 hence the uniform boundedness of solutions. The difficulties connected with the lack of diffusion term will be overcome by using ideas presented in \cite{HIlhorst-Matano-Nguyen-Weber}.
More precisely,
applying the rearrangement theory, 
we introduce the equi-measurable rearrangement $u^\sharp$ and show that it is the solution of a one-dimensional problem
$(P^\sharp)$ (see section 3). Since the orbit $\{u^\sharp(t) : t \ge 0\}$ is bounded in $BV(\Omega^\sharp)$, 
where $\Omega^\sharp:=(0, |\Omega|) \subset \R$, it is relatively compact in
$L^1(\Omega^\sharp)$. We then deduce the relative compactness of solution orbits of Problem $(P)$, by using the fact that
\begin{equation}\label{eq:24}
\|u(t)-u(\tau)\|_{L^1(\Omega)}=\|u^\sharp(t)-u^\sharp(\tau)\|_{L^1(\Omega^\sharp)}.
\end{equation}
Note that the inequality $\|u(t)-u(\tau)\|_{L^1(\Omega)} \ge \|u^\sharp(t)-u^\sharp(\tau)\|_{L^1(\Omega^\sharp)}$ 
follows from a general property of the \tn{rearrangement theory}. The important point is that \eqref{eq:24} involves an equality.

\medskip

An other advantage of considering Problem $(P^\sharp)$ is that the differential equations in $(P^\sharp)$ and $(P)$ have the same form. Therefore we will study the $\omega$-limit set for Problem $(P^\sharp)$ rather than for Problem $(P)$.
Although $(P^\sharp)$ possesses many stationary solutions, 
the one-dimensional structure of Problem $(P^\sharp)$ allows us to characterize its $\omega$-limit set, and then deduce results for that of $(P)$.

\medskip

{The organization of this article is as follows:} In section 2, we prove the global existence and uniqueness of the solution 
as well as its uniform boundedness. Next in section 3, 
we recall and apply results from the arrangement theory presented in \cite{HIlhorst-Matano-Nguyen-Weber} to obtain
the \tn{relative compactness} of the solution in $L^1(\Omega)$.
In section 4, we prove that Problem $(P)$ possesses Lyapunov functionals and use them 
together with the relative compactness of the solution to show that $\omega(u_0)$ is nonempty and consists of stationary solutions. Moreover, these stationary solutions are step functions. More precise properties of these functions are given in Theorems \ref{chap3:3chapchapomega:limit}  and \ref{propo:omega:2:rangbuoc}.
In section 5, we suppose that one of the hypotheses $\bf (H_1)$ or $\bf (H_3)$ holds and prove that $\omega(u_0)$ 
only contains one element.  

\tnr{In the case that Hypothesis $\bf (H_2)$ is satisfied, the structure of the $\omega$-limit set becomes more complicated than in the other cases since the solution can develop many transition layers. More precisely, as we will see in Theorem \ref{chap3:3chapchapomega:limit}, elements in the $\omega$-limit set may contain step functions taking three values $\{0, 1, \nu\}$ instead of the two values $\{1, \mu\}$ in the case $\bf (H_1)$ and $\{0, \xi\}$ in the case $\bf (H_3)$. As a consequence, it is more difficult to prove that the $\omega$-limit set contains a single element. We refer to our forthcoming paper \cite{HIlhorst-Laurencot-Nguyen} for a study in more details of the case $\bf(H_2)$ .}


\section{Existence and uniqueness of solutions of $(P)$}\label{chap3:sec2}
\subsection{Local existence}

First we prove the local Lipschitz property of the nonlocal nonlinear term $F$, given by \eqref{eq:defi:F},
  in the space $L^\infty(\Omega)$.
\begin{lemma}[Local Lipschitz continuity of $F$]\label{chap3:3chapchaptinhchat:lipchitz:23:08}
Let $v \in L^\infty(\Omega)$ be such that
 $\int_\Omega g(v(x))\,dx \neq 0$. Then there exist a $L^\infty(\Omega)$-neigbourhood  $\mathcal V$ of $v$ and  a constant $L>0$ such that 
\tn{$F(\widetilde v)$ is well-defined for all $\widetilde v \in \mathcal V$
 and that}
$$\|F(v_1)-F(v_2)\|_{L^\infty(\Omega)} \le L \|v_1-v_2\|_{L^\infty(\Omega)},$$
for all $v_1, v_2 \in \mathcal V$.
\end{lemma}

\begin{proof}
Since $g$ is continuous, the map $v \mapsto \int_\Omega g(v)$ is continuous from $L^\infty(\Omega)$ to $L^\infty(\Omega)$. It follows that 
there exist a constant $\alpha>0$ and a neighbourhood $\mathcal V$ of $v$ such that 
\begin{equation}\label{chap3:3chapchapgia:thief:cho:ngay:27:08}
 \left | \int_\Omega g(\widetilde v) \right | \ge \alpha \mbox{~~for all~~} \widetilde v \in \mathcal V.
\end{equation}
Without loss of generality, we may choose 
$$\mathcal V:= \{\widetilde v   \in L^\infty(\Omega):\| \widetilde v-v \|_{L^\infty(\Omega)} \le \e\},$$
for a constant $\e>0$ small enough. We set 
$$\bar c:= \|v\|_{L^\infty(\Omega)}+\e, \quad f(s):=g(s)p(s),$$
and
\begin{align*}
&K:= \max\big\{\sup_{[-\bar c, \bar c]} |f(s)|, \sup_{[-\bar c, \bar c]} |g(s)|, \sup_{[-\bar c, \bar c]} |f'(s)|, \sup_{[-\bar c, \bar c]} |g'(s)|\big\}.
\end{align*}
Then the following properties \tnr{hold} and will be used later:  \tnr{for} all $v_1, v_2 \in \mathcal V$,
\begin{equation}\label{chap3:3chapchapbdt:for:f:lop:24:08}
\|f(v_1)-f(v_2)\|_{L^\infty(\Omega)} \le K \|v_1-v_2\|_{L^\infty(\Omega)},
\end{equation}
and
\begin{equation*}\label{chap3:3chapchapbdt:for:h:lop:24:08}
\|g(v_1)-g( v_2)\|_{L^\infty(\Omega)} \le K \|v_1-v_2\|_{L^\infty(\Omega)}.
\end{equation*}

We have
\begin{align*}
F(v_1)-F(v_2)&=[ f(v_1)- f(v_2)] - \left[ g(v_1) \dfrac{\displaystyle \int_\Omega  f(v_1)}{\displaystyle \int_\Omega g(v_1)}-  
		 g(v_2) \dfrac{\displaystyle \int_\Omega  f(v_2)}{\displaystyle \int_\Omega g(v_2)} \right]\\								
		 &=[ f(v_1)- f(v_2)] - \dfrac{ g(v_1) \displaystyle \int_\Omega  f(v_1) \displaystyle \int_\Omega g(v_2) - g(v_2)  \displaystyle \int_\Omega  f(v_2)  \int_\Omega g(v_1)}
		{ \displaystyle \int_\Omega g(v_1) \int_\Omega g(v_2)\displaystyle}	\\
										&=:A_1-\frac{A_2}{A_3},
\end{align*}
where 
$$A_1:= f(v_1)- f(v_2),$$
$$A_2:= g(v_1) \displaystyle \int_\Omega  f(v_1) \displaystyle \int_\Omega g(v_2) -  g(v_2) \int_\Omega  f(v_2) \int_\Omega g(v_1),$$
and 
$$A_3:= \int_\Omega g(v_1) \int_\Omega g(v_2).$$
In the sequel, we estimate $A_1$, $A_2$ and $A_3$. 
First the inequality \eqref{chap3:3chapchapbdt:for:f:lop:24:08} yields
\begin{equation}\label{chap3:3chapchapba:dang:thuc:so:1:cho:chung:minh:G:v}
\|A_1\|_{L^\infty(\Omega)}  \le K \|v_1-v_2\|_{L^\infty(\Omega)}.
\end{equation}
Next we write $A_2$ as
\begin{multline*}
A_2= g(v_1) \displaystyle \int_\Omega  f(v_1) \displaystyle \int_\Omega g(v_2) - g(v_2)\displaystyle \int_\Omega  f(v_1) \displaystyle \int_\Omega g(v_2)\\
+ g(v_2)\displaystyle \int_\Omega  f(v_1) \displaystyle \int_\Omega g(v_2) - g(v_2)\displaystyle \int_\Omega  f(v_2) \displaystyle \int_\Omega g(v_2)\\
+ g(v_2)\displaystyle \int_\Omega  f(v_2) \displaystyle \int_\Omega g(v_2)
 -  g(v_2) \displaystyle \int_\Omega  f(v_2) \int_\Omega g(v_1),
\end{multline*}
or equivalently,
\begin{multline*}
A_2=[ g(v_1) -  g(v_2)] \displaystyle \int_\Omega  f(v_1) \displaystyle \int_\Omega g(v_2)\\
+ g(v_2)\displaystyle \int_\Omega [ f(v_1)- f(v_2)] \displaystyle \int_\Omega g(v_2) \\
+ g(v_2)\displaystyle \int_\Omega  f(v_2) \displaystyle \int_\Omega [g(v_2)- g(v_1)],
\end{multline*}
which in turn implies that
\begin{equation}\label{chap3:3chapchapbdt:cho:A:2:chung:minh:lip:G:24:08}
\|A_2\|_{L^\infty(\Omega)} \le    3 K^3|\Omega|^2 \|v_1-v_2\|_{L^\infty(\Omega)}. 
\end{equation}
As for the term $A_3$, we apply \eqref{chap3:3chapchapgia:thief:cho:ngay:27:08} to obtain
\begin{equation}\label{chap3:3chapchapbdt:cho:A:3:chung:minh:lip:G:24:08}
|A_3|  \ge \alpha^2>0.
\end{equation}
Combining \eqref{chap3:3chapchapba:dang:thuc:so:1:cho:chung:minh:G:v}, \eqref{chap3:3chapchapbdt:cho:A:2:chung:minh:lip:G:24:08} and \eqref{chap3:3chapchapbdt:cho:A:3:chung:minh:lip:G:24:08}, we deduce that
$$\|F(v_1)-F(v_2)\|_{L^\infty(\Omega)} \le \Big(K+\frac{3 K^3|\Omega|^2}{\alpha^2}\Big) \,\|v_1-v_2\|_{L^\infty(\Omega)}.$$
This completes the proof of Lemma \ref{chap3:3chapchaptinhchat:lipchitz:23:08}.
\end{proof}

\begin{proposition}\label{lem:local:existence}
Let $u_0 \in L^\infty(\Omega)$ satisfy $\int_\Omega g(u_0) \neq 0$. Then Problem $(P)$ has a unique local-in-time solution. Moreover, we have
\begin{equation}\label{eq:mass-conservation}
\int_\Omega u(x,t)\,dx=\int_\Omega u_0(x)\,dx \mbox{~~for all~~} t \in [0, T_{max}(u_0)),
\end{equation}
where $T_{max}(u_0)$ denotes the maximal time interval of the existence of solution. 
\end{proposition}
\begin{proof}
Since $F$ is locally Lipschitz continuous in $L^\infty(\Omega)$, the local existence follows from the standard theory of ordinary differential equations. We now prove \eqref{eq:mass-conservation}.
Integrating the differential equation in Problem $(P)$ from $0$ to $t$, we obtain
$$u(t)-u_0=\int_0^t u_t{(s)\,ds} =\int_0^t F(u{(s))\,ds}.$$
It follows that
 $$\int_\Omega  u(x,t)\,dx-\int_\Omega  u_0(x)\,dx=\int_0^t \int_{\Omega} F(u) {\,dxds}=0,$$
 where the last identity holds since 
 $$\int_{\Omega} F(u) {\,dx}=0.$$
 This completes the proof of the \tnr{proposition}.
\end{proof}

\begin{lemma}\label{lem:blow:up:time}
If $T_{max}(u_0) <\infty$ and $\limsup_{t \uparrow T_{max}(u_0)} \|F(u(t))\|_{L^\infty(\Omega)}<\infty$, then 
$u(T_{max}(u_0)-):=\lim_{t \uparrow T_{max}(u_0)} u(t)$ exists  in $L^\infty(\Omega)$ and  
$$\int_\Omega g(u(T_{max}(u_0)-))=0.$$
\end{lemma}

\begin{proof}
For simplicity we write $T_{max}$ instead of $T_{max}(u_0)$. Set
$$M:=\limsup_{t \uparrow T_{max}} \|F(u(t))\|_{L^\infty(\Omega)} < \infty.$$ 
Then there exists $0<T<T_{max}$ such that
$$ \| F(u(t))\|_{L^\infty(\Omega)} \le 2M \mbox{~~for all~~}t \in [T, T_{max}).$$
Consequently, for any $\tn{t , t'} \in [T, T_{max})$, with $t <t'$,
we have
$$\|u(t) -u(t')\|_{L^\infty(\Omega)} \le \int_t^{t'}\|\tnr{F(u(s))\|_{L^\infty(\Omega)} \,ds} \le \tn{ 2M} |t-t'|.$$
Thus $\{u(t)\}$ is a Cauchy sequence so
that the limit $u(T_{max}-):=\lim_{t \uparrow T_{max}} u(t)$ exists  in $L^\infty(\Omega)$. If $\int_\Omega g(u(T_{max}-)) \neq 0$, then, by Lemma \ref{lem:local:existence}, we can extend the solution on $[T_{max}, T_{max}+\delta)$, with some $\delta >0$, which contradicts the definition to $T_{max}$. This completes the proof of the lemma.
\end{proof}


\subsection{Global solution}

In this subsection, we fix $u_0 \in L^\infty(\Omega)$ satisfying $\int_\Omega g(u_0) \neq 0$ and denote by $[0, T_{max})$ the maximal time interval of the existence of solution.
Set
\begin{equation}\label{definition:lamda:t}
\lambda(t)=\frac{\displaystyle\int_\Omega g(u)p(u)}{\displaystyle\int_\Omega g(u)} \mbox{~~for all~~} t \in [0, T_{max}),
\end{equation}
and study solutions $Y(t;s)$ of the following auxiliary problem:
\begin{equation}
(ODE) \,\,
\begin{cases}
\dot{Y}= g(Y)p(Y)-g(Y)\lambda(t), \quad t > 0,\vspace{6pt}\\
Y(0)=s,
\end{cases}
\end{equation}
where $\dot{Y}:=dY/dt$. We remark that the function $u$ satisfies
\begin{equation}\label{problem:ODE:22:6:3:14}
u(x,t)=Y(t;u_0(x)) \mbox{~~for a.e.~~} x \in \Omega \mbox{~~and all~~}t \in [0, T_{max}).
\end{equation}

\begin{lemma}\label{lem:monotonicity:of:Y}
Let $\widetilde s <s$ and let $0<T<T_{max}$. Assume that Problem \tn{$(ODE)$}
possesses the
solutions $Y(t;\widetilde s), Y(t;s) \in C^1([0, T])$, respectively.
Then
\begin{equation}\label{ine:monotone:Y}
Y(t;\widetilde s)<Y(t;s) \mbox{~~for all~~} t  \in [0, T].
\end{equation}
\end{lemma}

\begin{proof}
Since $Y(0;\widetilde s)=\widetilde s<s=Y(0;s)$, the assertion follows immediately from the backward uniqueness of solution of $(ODE)$.
\end{proof}

\begin{theorem}\label{thm:existence:boundedness}
Assume that one of the hypotheses $\bf (H_1), (H_2), (H_3)$ holds.
Then Problem $(P)$ possesses a global solution $u \in C^1([0, \infty); L^\infty(\Omega))$.
Moreover:
\begin{enumerate}[label=\emph{(\roman*)}]
\item If  $\bf (H_1)$ holds, then for all $t \ge 0$,
\begin{equation}\label{ine:h1}
1 \le u(x,t) \le \esssup_{\Omega}u_0 \mbox{~~for a.e.~~}x\in \Omega.
\end{equation}
\item If $\bf (H_2)$ holds, then for all $t \ge 0$,
\begin{equation}\label{ine:h2}
0 \le u(x, t) \le 1\mbox{~~for a.e.~~}\tnr{x\in \Omega}.
\end{equation}
\item If  $\bf (H_3)$ holds, then for all $t \ge 0$,
$$\essinf_{\Omega} u_0 \le u(x,t) \le 0 \mbox{~~for a.e.~~}x\in \Omega.$$
\end{enumerate}
\end{theorem}

\begin{proof}
For simplicity, we set
$$a:=\essinf_{\Omega} u_0, \qquad b:=\esssup_{\Omega} u_0.$$
We only prove (i) and (ii). The proof of (iii) is similar to that of (i).

(i) First, we show that \eqref{ine:h1} holds as long as the
solution $u$ exists and then deduce the global existence from Lemma \ref{lem:blow:up:time}. 
Let $Y(t;s)$ be the solution of $(ODE)$.
We remark that $b \ge 1$ and that $Y(t,1) \equiv 1$ for all $t\in [0, T_{max})$. The monotonicity of $Y(t;s)$
in $s$ implies that as long as $u, Y(t;b)$ both exist
\begin{equation}\label{eqs:need:fortheorem}
1 \equiv Y(t,1) \leq Y(t;u_0(x))=u(x,t) \leq Y(t;b)\quad\ a.e.\ \ x\in\Omega.
\end{equation}
The first inequality above implies the first inequality of \eqref{ine:h1} as long as the solution $u$ exists. 
It remains to prove the second inequality of \eqref{ine:h1}.  To that purpose,  it suffices to show that
\begin{equation}\label{ab-Y}
  Y(t;b)\leq b
\end{equation}
as long as the solution $Y(t;b)$
exists. In view of \eqref{eqs:need:fortheorem}, we have $1 \le u(x,t) \le Y(t;b)$. Then
the definition of $g$ and the monotonicity of $p$ imply that
$$g(Y(t;b)) \le 0, \quad g(u(x,t)) \le 0, \quad p(Y(t,b)) \ge p(u(x,t)),$$
for a.e. $x \in \Omega$. \tnr{These properties, together with the definition of $\lambda(t)$ in \eqref{definition:lamda:t}, imply that}
\tnr{\begin{align*}
\dot Y(t;b)&=g(Y(t,b))(p(Y(t;b))-\lambda(t))\\
&=g(Y(t,b))\left(p(Y(t;b))-\frac{\displaystyle\int_\Omega g(u(x,t))p(u(x,t))\,dx}{\displaystyle\int_\Omega g(u(x,t))\,dx} \right)\\
&=g(Y(t,b)) \frac{\displaystyle\int_\Omega g(u(x,t))[p(Y(t;b)-p(u(x,t))]\,dx}{\displaystyle\int_\Omega g(u(x,t))\,dx} \le 0.
\end{align*}}
Hence
$$Y(t,b) \le Y(0;b)=b,$$
which completes the proof of \eqref{ab-Y}. Thus \eqref{ine:h1} is satisfied 
as long as the solution $u$ exists.

Next we show that the solution $u$ exists globally. Suppose, by contradiction,  that  $T_{max}<\infty$. We have
for all $t \in [0, T_{max})$,
$$| \lambda(t) | \le \dfrac{\int_\Omega |g(u)p(u)|}{\left |\int_\Omega g(u) \right|}=\dfrac{\int_\Omega |g(u)|\,|p(u)|}{\int_\Omega |g(u)| } \le \max\{|\tn{p(1)}|, |p(b)|\}.$$
It follows that there exists $C>0$ such that $\|F(u(t))\|_{L^\infty(\Omega)} \le C$  for all $t \in [0, T_{max})$. By Lemma \ref{lem:blow:up:time}, $u(T_{max}-):=\lim_{t \uparrow T_{max}} u(t)$ exists  in $L^\infty(\Omega)$ and  
$$\int_\Omega g(u(T_{max}-))=0.$$ 
Since $u(x,t) \ge 1$ for a.e $x \in \Omega, t \in [0, T_{max})$, $u(x,T_{max}-) \ge 1$ for a.e. $x \in \Omega$. Hence $\int_\Omega g(u(T_{max}-))=0$ if and only if $u(x,T_{max}-) \equiv  1$.
The mass conservation property (cf. \eqref{eq:mass-conservation}) yields $\int_\Omega u_0  =|\Omega|$. Hence $u_0(x)=1$ for a.e $x \in \Omega$. This contradicts Hypothesis $\bf (H_1)$ so that $T_{max}=\infty$.

(ii)  Since  $Y(t,1) \equiv 1$, $Y(t,0) \equiv 0$, we deduce that
$$0 \equiv Y(t,0) \leq Y(t;u_0(x))=u(x,t) \le Y(t,1) \equiv 1\quad\ a.e.\ \ x\in\Omega.$$
This implies \eqref{ine:h2} as long as the solution $u$ exists. We now prove that $T_{max}=\infty$. Indeed, 
suppose, by contradiction,  that  $T_{max}<\infty$. Since $0 \le u(x,t) \le 1$ for a.e. $x \in \Omega$, and all $t \in [0, T_{max})$, $g(u(x,t)) \ge 0$ for a.e. $x \in \Omega$, and all $t \in [0, T_{max})$. 
Therefore
$$| \lambda(t) | \le \dfrac{\int_\Omega | g(u) p(u)|}{\left |\int_\Omega g(u) \right|}=\dfrac{\int_\Omega g(u) \, |p(u)|}{\int_\Omega g(u) } \le \max\{|p(0)|, |p(1)|\},$$
for all $t \in [0, T_{max})$.
It follows that there exists $C>0$ such that $\|F(u(t))\|_{L^\infty(\Omega)} \le C$  for all $t \in [0, T_{max})$. By Lemma \ref{lem:blow:up:time}, $u(T_{max}-):=\lim_{t \uparrow T_{max}} u(t)$ exists  in $L^\infty(\Omega)$ and  
$$\int_\Omega g(u(T_{max}-))=0.$$ 
This implies that
$u(T_{max}-)$ only takes two values $0$ and $1$. Or equivalently,
$Y(T_{max}-;u_0(x)) $ only takes two values $0$ and $1$. Thus
the backward uniqueness of the solution of  the initial value problem $(ODE)$ implies that
$u_0(x)$ only takes two values $0$ and $1$; hence $\int_\Omega g(u_0) =0$.
This contradicts Hypothesis $\bf (H_2)$ so that $T_{max}=\infty$.
\end{proof}

The result below follows from the proof of Theorem \ref{thm:existence:boundedness}.

\begin{corollary}\label{cor:boundedness:lamda(t)}
Assume that one of the hypotheses $\bf (H_1), (H_2), (H_3)$ holds and  let $\lambda(t)$ be defined by  \eqref{definition:lamda:t}. Then there exists $C>0$ such that
$$|\lambda(t)| \le C $$
$\mbox{for all~~} t \in[0, \infty).$
\end{corollary}

\section{Boundedness of the solution and one-dimensional associated problem $(P^\sharp)$}
All the results in this section are similar to those of \cite[Section 3]{HIlhorst-Matano-Nguyen-Weber}. 
We recall and state some important results. 
Let $w$ be a function from $\Omega$ to $\R$ and let $\Omega^\sharp:=(0, |\Omega|) \subset \R$. The distribution function of $w$ is given by
$$\mu_w(s):=|\{x \in \Omega: w(x)>s\}|.$$
\begin{definition}
The (one-dimensional) decreasing rearrangement of $w$, denoted by $w^\sharp$, is defined on $\overline \Omega^\sharp=[0, |\Omega|]$ by
\begin{equation}\label{definition:w:sharp}
\begin{cases}
&w^\sharp(0):=\esssup (w)\\
&w^\sharp(y)=\inf\{ s:   \mu_w(s )<y\}, \quad y>0.
\end{cases}
\end{equation}
 \end{definition}

 \begin{remark}\label{chuy:tinhchat:w:shap}
The function $w^\sharp$ is nonincreasing on $\Omega^\sharp$ and we have
$\mu_w(s)=\mu_{w^\sharp}(s)$ for all $s \in \R$. Moreover, if $a \le w(x) \le b$ a.e. $x \in \Omega$, then
$$a \le w^\sharp(y) \le b \mbox{~~for all~~} y \in \Omega^\sharp.$$
 \end{remark}

\begin{theorem}\label{chap3:3chapchapthm:ws:star}
Let one of the hypotheses $\bf (H_1), (H_2), (H_3)$ hold. 
We define
\begin{equation}\label{defition:u:sharp}
u^\sharp(y,t):=(u(t))^\sharp(y) \mbox{~~on~~} \Omega^\sharp \times [0,+\infty).
\end{equation}
Then
 $u^\sharp$ is the unique solution \tn{in} $C^1([0,\infty); L^\infty(\Omega^\sharp))$ of Problem $(P^\sharp)$
\begin{equation*}
 (P^\sharp) \ \ \left\{
\begin{array}{ll}
\tn{\dfrac{dv}{dt}=g(v)p(v)-g(v) \dfrac{\displaystyle \int_\Omega  g(v)p(v)}{\displaystyle \int_\Omega g(v)}} \quad \quad & t > 0, \vspace{8pt}\\
v(0)=u_0^\sharp. 
\end{array}
\right.
\end{equation*}
Moreover, for all $t \ge 0$,
\begin{equation}\label{chap3:6:09:bdt:cho:nghiem:u:sao:ha:ha}
u^\sharp(y,t)=Y(t;u_0^\sharp(y)) \mbox{~~for a.e.~~}y\in  \Omega^\sharp,
\end{equation} 
and the assertions (i), (ii), (iii) of Theorem 
\ref{thm:existence:boundedness} hold for the function $u^\sharp$.
\end{theorem}

 \begin{lemma}[{\cite[Lemma 3.7]{HIlhorst-Matano-Nguyen-Weber}}] \label{chuyenthanhbode:quenheeeee:lemmme:u-shapr:u}
 Let $u$ be the solution of $(P)$ with $u_0 \in L^\infty(\Omega)$ and let $u^\sharp$ be as in
 \eqref{defition:u:sharp}. Then
 \begin{align}\label{danthuc123:u-sparp-u:13:3:14}
\|u^\sharp(t)-u^\sharp(\tau)\|_{L^1(\Omega^\sharp)}=\|u(t)-u(\tau)\|_{L^1(\Omega)},
\end{align}
for any $t , \tau \in [0, \infty)$.
\end{lemma}

\begin{corollary}[{\cite[Corollary 3.9]{HIlhorst-Matano-Nguyen-Weber}}] \label{quenheeeee:lemmme:u-shapr:u}
Let $\{t_n\}$ be a sequence of positive numbers such that $t_n \to \infty$ as $n \to \infty$. Then
 the following statements are equivalent
\begin{enumerate}[label=\emph{(\alph*)}]
\item $u^\sharp(t_n) \to \psi$ in $L^1(\Omega^\sharp)$ as $n \to \infty$
for some $\psi\in L^1(\Omega^\sharp)$;
\item $u(t_n) \to \varphi$ in $L^1(\Omega)$ as $n \to \infty$
for some $\varphi\in L^1(\Omega)$ with $\varphi^\sharp=\psi$.
\end{enumerate}
\end{corollary}

\tn{The following proposition follows from similar results in \cite[Lemma 3.5 and Proposition 3.10]{HIlhorst-Matano-Nguyen-Weber}.}
\begin{proposition}\label{boundedness:solution:P:15} Let one of the hypotheses $\bf (H_1), (H_2), (H_3)$ hold. Then
$\{ u(t): t \ge 0 \}$  is relatively compact in $L^1(\Omega)$ and the set $\{ u^\sharp(t): t \ge 0 \}$  is relatively compact in $L^1(\Omega^\sharp)$.
\end{proposition}


\section{Lyapunov functional and $\omega$-limit set for $(P)$}
We define three Lyapunov functionals according to whether the initial function satisfies 
either Hypothesis $\bf (H_1) , (H_2)$ or $\bf (H_3)$.
More precisely, we define for $i=1, 2, 3$ the functional $E_i$ by
\begin{equation}\label{defi:lyapunove:function}
E_i(u)=(-1)^{i+1}\int_{\Omega} \mathcal P(u),
\end{equation}
where 
$$\mathcal P(s)=\int_0^s p(\tau)\,d \tau.$$

\begin{lemma}[Lyapunov functional]\label{chap3:3chapchapLyapunov-functional}
Assume that the hypotheses ${\bf (H_i)}$ holds either for $i=1$, or for $i=2$, or for $i=3$. Then
\begin{enumerate}[label=\emph{(\roman*)}]
\item There exists $C>0$ such that for all $\tau_2 > \tau_1 \ge 0$, 
\begin{align*}
E_i(u(\tau_2))-E_i(u(\tau_1))&= (-1)^{i+1}\int_{\tau_1}^{\tau_2}\int_{\Omega} g(u)  (p(u)-\lambda(t))^2 \,dxdt\\
&\le -C\int_{\tau_1}^{\tau_2}\int_{\Omega} |u_t|^2  \,dxdt \le 0.
\end{align*} 
\item $E_i(u(\cdot))$ is continuous and non increasing on $[0, \infty)$, and \tn{the limit $E_{i\infty}:= \lim_{t \to \infty} E_i(u(t))$ exists}. 
\end{enumerate}
\end{lemma}

\tn{\begin{remark}\label{rem:unique;e:infty}
	Note that the solution orbit $\{u(t): t \ge 0\}$ is uniquely defined by the initial function $u_0$. Hence
 $E_{i\infty}$---the limit of the Lyapunov functional along the solution orbit---is also uniquely defined by the initial function.
\end{remark}}

\begin{proof}[\bf Proof of Lemma \ref{chap3:3chapchapLyapunov-functional}]
(i) We only give the proof for the case $i=1$. We have
\begin{align*}
\frac{d}{dt} E_1(u(t))&=\frac{d}{dt} \int_{\Omega} \mathcal P(u) \,dx
= \int_{\Omega} p(u)u_t \,dx.
\end{align*}
Since 
$$\int_{\Omega} u_t \,dx=0,$$
it follows that
\begin{align}
\frac{d}{dt} E_1(u(t)) &= \int_{\Omega} (p(u)-\lambda(t)) u_t \,dx\notag\\
&=\int_{\Omega} g(u)(p(u)-\lambda(t))^2 \,dx.  \label{bdt:cho:lyapunov18:02:2014}
\end{align}
Set 
$$C=-\frac{1}{\min_{s \in [1,\esssup u_0]} g(s)}>0.$$
Then, since for all $t\ge 0$,
$$1 \le u(t) \le \esssup u_0 \mbox{~~a.e. in~~}\Omega,$$
we have, for all $t \ge 0$,
$$-\frac{1}{C}  \le g(u(t)) \le 0 \mbox{~~a.e. in~~}\Omega.$$
As a consequence, for all $t \ge 0$,
$$g(u(t)) \le - C g^2(u(t))  \mbox{~~a.e in~~}\Omega.$$
Substituting this inequality into \eqref{bdt:cho:lyapunov18:02:2014} yields
\begin{align*}
\frac{d}{dt} E_1(u(t))  &\le-C\int_\Omega g^2(u)(p(u)-\lambda(t))^2 \,dx\\
&= -C \int_\Omega |u_t|^2 \,dx \le 0,
\end{align*}
which proves (i).

(ii) As a consequence of (i), $E_i(u(\cdot))$ is continuous and nonincreasing. Moreover, $E_i$ is bounded from below. Therefore there exists  the limit of $E_i(u(t))$ as $t \to \infty$, which completes the proof of (ii).
\end{proof}

\begin{proposition}\label{cautruc:omega:limit}
Assume that one of the hypotheses ${\bf (H_i)}, (i=1,2,3)$ holds. Then 
\begin{enumerate}[label=\emph{(\roman*)}]
\item $\omega(u_0)$ a nonempty set in $L^1(\Omega)$.
\item\tn{Let $E_{i\infty}$ be given in Lemma \ref{chap3:3chapchapLyapunov-functional}, we have}
\begin{equation}\label{eq:rangbuoc:omega}
\tn{E_i(\varphi)=E_{i\infty} \ \ \mbox{for all} \ \ \varphi \in \omega(u_0).}
\end{equation}
\tn{In other words,  $E_i(\cdot)$ is constant on $\omega(u_0)$.} 
\item Any element $\varphi \in \omega(u_0)$ either satisfies $\int_\Omega g(\varphi)\,dx=0$ or is a stationary solution of Problem $(P)$. 
\end{enumerate}
\end{proposition}

\begin{proof}
(i) follows from Proposition \ref{boundedness:solution:P:15}. (ii) 
\tn{Let $\varphi \in \omega(u_0)$ and let $t_n\to \infty$ be a sequence such that
\begin{equation*}
u(t_n) \to \varphi\mbox{~~in~~}L^1(\Omega) \mbox{~~as~~} n\to\infty.
\end{equation*}
Since $\{u(x,t_n)\}$ is uniformly bounded, the convergence $u(t_n) \to \varphi$ implies
$E_i(u(t_n)) \to E_i(\varphi)$ as $n \to \infty$. Hence in view of  Lemma \ref{chap3:3chapchapLyapunov-functional}\,(ii) we have} 
$$\tn{E_i(\varphi)=\lim_{n \to \infty} E_i(u(t_n))=E_{i\infty}}.$$
(iii) Assume that 
$$\int_\Omega g(\varphi(x))\,dx\neq 0;$$
 we show below that $\varphi$ is a stationary solution of Problem $(P)$. 
Let $\{t_n\}$ be a sequence such that $t_n\to \infty$ and
\begin{equation}\label{chap3:3chapchapsuhoitu:point:omega:limit:set}
 u(t_n) \to \varphi\mbox{~~in~~}L^1(\Omega) \mbox{~~as~~} n\to\infty.
\end{equation}
It follows from Lemma \ref{chap3:3chapchapLyapunov-functional} that
$$\int_0^{+\infty} \int_{\Omega} |u_t|^2 \,dxdt \le \frac{1}{C} (E_i(u_0)-\lim_{t \to\infty}E_i(u(t))<+\infty.$$ 
Thus 
$$\lim_{n \to \infty} \int_{t_n}^{t_n+1} \int_{\Omega} |u_t|^2 \,dxdt=0.$$ 
It follows that for $t \in [0, 1]$,
\begin{align*}
\|u(t_n+t) -\varphi\|_{L^1(\Omega)}&\le  \|u(t_n+t) -u(t_n)\|_{L^1(\Omega)}+\|u(t_n) -\varphi\|_{L^1(\Omega)}\\
&\le \int_{t_n}^{t_n+t} \|u_t\|_{L^1(\Omega)} +\|u(t_n) -\varphi\|_{L^1(\Omega)}\\
&\le t^{\frac{1}{2}} |\Omega|^{\frac{1}{2}} \left(\int_{t_n}^{t_n+t}\int_\Omega |u_t|^2 \,dxdt\right )^{\frac{1}{2}} +\|u(t_n) -\varphi\|_{L^1(\Omega)} \to 0
\end{align*}
as $n \to \infty$. Hence for all $t\in [0, 1]$,
 $$u(t_n+t) \to \varphi\mbox{~~in~~}L^1(\Omega) \mbox{~~as~~} n\to\infty.$$
Set $f(s)=g(s)p(s)$; then
the uniform boundedness of $u$ (cf. Theorem \ref{thm:existence:boundedness})
implies that
for all $t\in [0, 1]$,
$$ f(u(t_n+t)) \to f(\varphi), \quad g(u(t_n+t)) \to g(\varphi) \mbox{~~in~~}L^1(\Omega).$$
as $n\to\infty$. Note that 
$\int_\Omega g(\varphi) \neq 0$ implies
$$\lambda(t_n+t):=\dfrac{\int_\Omega f(u(t_n+t))}{\int_\Omega g(u(t_n+t))} \to \dfrac{\int_\Omega f(\varphi)}{\int_\Omega g(\varphi)}\mbox{~~as~~} n\to\infty.$$
It follows that for all $t \in [0, 1]$
$$\int_\Omega F(u(t_n+t))\,dx \to \int_\Omega F(\varphi) \,dx\mbox{~~as~~}n \to\infty,$$
where $F$ is defined by \eqref{eq:defi:F}. On the other hand, the uniform boundedness of $u(x,t)$ and Corollary \ref{cor:boundedness:lamda(t)} imply that $F(u(x,t))$ is uniformly bounded. Thus 
$$\int_\Omega F^2(u(t_n+t)) \,dx \to \int_\Omega F^2(\varphi) \,dx\mbox{~~as~~}n \to\infty,$$
for all $t \in [0, 1]$.
By Lebesgue's dominated convergence theorem, we have
$$\int_{0}^1\int_\Omega F^2(u(t_n+t)) \,dxdt\to \int_{0}^1\int_\Omega F^2(\varphi) \,dxdt\mbox{~~as~~}n \to\infty.$$
Since
$$\int_{0}^1\int_\Omega F^2(u(t_n+t)) \,dxdt=\int_{t_n}^{t_n+1} \int_{\Omega} F^2(u(t)) \,dxdt=\int_{t_n}^{t_n+1} \int_{\Omega} |u_t|^2 \,dxdt\to 0,$$
as $n \to \infty$, it follows that
$$\int_{0}^1\int_\Omega F^2(\varphi)\,dxdt=0.$$
This yields 
$$F(\varphi)=0 \mbox{~~a.e. in~~}\Omega,$$
or equivalently,
$$g(\varphi)\left [p(\varphi) - \dfrac{\displaystyle \int_{\Omega} g(\varphi)p(\varphi)}{\displaystyle \int_{\Omega} g(\varphi)}\right ] =0 \mbox{~~a.e. in~~}\Omega.$$
Therefore
$\varphi$ is a stationary solution of Problem $(P)$.
\end{proof}

\tn{In the two following theorems, we obtain a more precise description of the  elements in the $\omega$-limit set.}
\begin{theorem}\label{chap3:3chapchapomega:limit} 
Assume that one of hypotheses ${\bf (H_i)}, (i=1,2,3)$ holds and let $\varphi \in \omega(u_0)$. Then:
\begin{enumerate}[label=\emph{(\roman*)}]
\item If $\bf (H_1)$ holds, then $ 1 \le \varphi \le \esssup_\Omega  u_0$ and $\varphi$ is a step function. More precisely,
$$\varphi= \mu \chi_{A_1}+  \chi_{\Omega\setminus {A_1}},$$
where $\mu >1$, $A_1 \subset \Omega, |A_1|\neq 0.$

\item If $\bf (H_2)$ holds, then $ 0 \le \varphi \le 1$ and $\varphi$ is a step function. More precisely,
$$\varphi= \chi_{A_1}+  \nu  \chi_{{A_2}},$$
where $0<\nu <1$, $A_1, A_2 \subset \Omega,$ with $A_1 \cup A_2 \subset \Omega$ and \tn{$A_1 \cap A_2=\emptyset$.}

\item If $\bf (H_3)$ holds, then $ \essinf_\Omega u_0 \le \varphi \le 0$ and $\varphi$ is a step function. More precisely,
$$\varphi= \xi\chi_{A_1},$$
where $\xi <0$, $A_1 \subset \Omega, |A_1| \neq 0.$
\end{enumerate}
\end{theorem}

\begin{proof}
We only prove (i); the proofs of (ii) and (iii) are similar.
Since for all $t \ge 0$, 
$$1 \le u(t) \le \esssup_\Omega u_0 \mbox{~~a.e. in~~}\Omega,$$
it follows that
$$ 1 \le \varphi \le \esssup_\Omega u_0 \mbox{~~a.e. in~~}\Omega.$$
Note that since $\int_\Omega \varphi = \int_\Omega u_0 >|\Omega|,$
$ \varphi \not\equiv 1$.
Therefore $\int_\Omega g(\varphi) < 0$. It follows from Proposition \ref{cautruc:omega:limit}\,(iii) that $\varphi$ is a stationary solution of $(P)$, namely
$$g(\varphi)\left [p(\varphi) - \dfrac{\displaystyle \int_{\Omega} p(\varphi) g(\varphi)}{\displaystyle \int_{\Omega} g(\varphi)}\right ] =0 \quad \mbox{~~a.e. in~~}\Omega,$$
which together with the monotonicity of $p$ yields 
$$\varphi= \mu \chi_{A_1}+  \chi_{\Omega\setminus {A_1}},$$
for some constant $\mu >1$ and $A_1 \subset \Omega$. Moreover $|A_1|\neq 0$ since $ \varphi \not\equiv 1$.
\end{proof}

\begin{theorem}\label{propo:omega:2:rangbuoc}
Assume that one of hypotheses ${\bf (H_i)}, (i=1,2,3)$ holds. \tn{Let $\varphi \in \omega(u_0)$ and let $E_{i\infty}$ be given in Lemma \ref{chap3:3chapchapLyapunov-functional}\,(ii).}
We set $m_0:=\int_\Omega u_0$.
Then
\begin{enumerate}[label=\emph{(\roman*)}]
\item If $\bf (H_1)$ holds, then 
$$\mu |A_1|+  |\Omega|-|A_1|=m_0, \quad \mathcal P(\mu) |A_1|+  \mathcal P(1)(|\Omega|-|A_1|)=\tn{E_{1\infty}}.$$
\item If $\bf (H_2)$ holds, then
$$|A_1|+  \nu |A_2| =m_0, \quad \mathcal P(1) |A_1|+  \mathcal P(\nu)|A_2|=\tn{-E_{2\infty}}.$$
\item If $\bf (H_3)$ holds, then
$$\xi |A_1|=m_0, \quad \mathcal P(\xi) |A_1|=\tn{E_{3\infty}}.$$
\end{enumerate}
\end{theorem}

\begin{proof}
	\tn{We only prove (i). The other cases can be proven in a similar way.
We apply \eqref{eq:rangbuoc:omega} for $i=1$ to obtain
\begin{equation}\label{eq:25:89}
\int_\Omega \mathcal P(\varphi)=E_{1\infty}.
\end{equation}
Hence (i) follows from \eqref{eq:25:89}, the mass conservation property and
Theorem \ref{chap3:3chapchapomega:limit}.}

\end{proof}


\section{Large time behavior of the solution of $(P)$}\label{chap3:sec6}
In this section, we only suppose Hypotheses $\bf (H_1), (H_3)$.

\begin{theorem}\label{uniqueness}
\begin{enumerate}[label=\emph{(\roman*)}]
\item Let $\bf (H_1)$ hold. Then $\omega(u_0)$
  only contains one element, \tn{denoted by $\varphi$}. Moreover
$\varphi$ is a step function of the form
$$\varphi= \mu \chi_{A_1}+  \chi_{\Omega\setminus {A_1}} \mbox{~~with~~}A_1 \subset \Omega, \mbox{~~and~~} \mu>1.$$
\item Let $\bf (H_3)$ hold. Then $\omega(u_0)$  only contains one element, \tn{denoted by $\psi$}. Moreover,
\tn{$\psi$} is a step function of the form
$$\tn{\psi}= \xi\chi_{A_1} \mbox{~~with~~}A_1 \subset \Omega, \mbox{~~and~~} \xi<0.$$
\end{enumerate}
\end{theorem}

\begin{proof}
We only prove (i). First we prove that $\omega(u_0^\sharp)$ only contains one element.
\tn{Note that (cf. Corollary \ref{quenheeeee:lemmme:u-shapr:u}) any element of $\omega(u_0^\sharp)$ has the form 
$\varphi^\sharp$ with $\varphi \in \omega(u_0)$.} 
 Since $\varphi^\sharp$ is non-increasing,  Theorem
\ref{chap3:3chapchapomega:limit} implies that there exist $0<a_1 \le |\Omega|, \mu>1$ such that
$$\varphi^\sharp=\mu \chi_{(0, a_1)}+  \chi_{(a_1, |\Omega|)}.$$
\tn{It follows from \eqref{eq:rangbuoc:omega}  and a standard property in the rearrangement theory (cf. \cite[Proposition 3.1\,(iii)]{HIlhorst-Matano-Nguyen-Weber}) that}
$$\tn{\int_{\Omega^\sharp} \mathcal P(\varphi^\sharp)=\int_{\Omega} \mathcal P(\varphi)=E_{1\infty}.}$$
\tn{Hence using the mass conservation property and recalling that $m_0:=\int_{\Omega} u_0$, we obtain} 
\begin{equation*}
\begin{cases}
\mu a_1+|\Omega|-a_1=m_0\\
\mathcal P(\mu) a_1+\mathcal P(1) (|\Omega|-a_1)=E_{1\infty},
\end{cases}
\end{equation*}
or equivalently,
\begin{equation}\label{sys:case:phuongtrinh}
\begin{cases}
(\mu-1) a_1=m_0-|\Omega|\\
(\mathcal P(\mu)- \mathcal P(1))a_1=E_{1\infty} -\mathcal P(1)|\Omega|.
\end{cases}
\end{equation}
Since we know the existence of a function $\varphi^\sharp$, we also know that the system \eqref{sys:case:phuongtrinh} possesses a solution. We show below that it is unique. Indeed, we deduce from \eqref{sys:case:phuongtrinh} that
\begin{equation}\label{sdfsfmpgojklkjly}
\dfrac{\mathcal P(\mu)-\mathcal P(1)}{\mu-1}=\dfrac{E_{1\infty} -\mathcal P(1)|\Omega|}{m_0-|\Omega|}.
\end{equation}
Set
$$\mathcal G(s)=\dfrac{\mathcal P(s)- \mathcal P(1)}{s-1}.$$
Then  \eqref{sdfsfmpgojklkjly} becomes
\begin{equation}\label{sdfsfmpgojklkjly222}
\mathcal G(\mu)=\dfrac{E_{1\infty} -\mathcal P(1)|\Omega|}{m_0-|\Omega|}.
\end{equation}
We use the monotonicity of $p$ to deduce that
\begin{align*}
\mathcal G'(s)&=\dfrac{p(s)(s-1)-(\mathcal P(s)- \mathcal P(1))}{(s-1)^2}\\
&=\dfrac{\displaystyle\int_1^s [p(s)-p(\tau)] d\tau}{(s-1)^2} >0 \mbox{~~for~~} s>1.
\end{align*}
Hence $\mathcal G$ is strictly increasing on $(1, \infty)$. 
\tn{It follows that the equation \eqref{sdfsfmpgojklkjly222}
admits at most one solution $\mu>1$. Furthermore, in view of Remark \ref{rem:unique;e:infty}, the right-hand-side of \eqref{sdfsfmpgojklkjly222} is uniquely defined by the initial function $u_0$. 
Thus $\mu$ is uniquely defined by the initial function.}
Therefore also $a_1$ is uniquely determined.
The knowledge of the constants $\mu$ and $a_1$ completely determines the stationary solution $\varphi^\sharp$,
so that $\omega(u^\sharp_0)$ only contains one element.

Next we show that $\omega(u_0)$ only contains one element. Since $\omega(u^\sharp_0)$ only contains one element, $u^\sharp(t)$ converges to $\varphi^\sharp$
as $t \to \infty$. Consequently,  $u^\sharp(t)$
 is a Cauchy sequence in $L^1(\Omega^\sharp)$. 
By Lemma \ref{chuyenthanhbode:quenheeeee:lemmme:u-shapr:u}, 
$u(t)$ is also a Cauchy sequence in $L^1(\Omega)$. This implies that 
$u(t)$ converges as $t \to \infty$ and hence
$\omega(u_0)$ only contains one element. 
\end{proof}

The following result is an immediate consequence of Theorem \ref{uniqueness} and the uniform boundedness of $u$.

\begin{corollary}\label{chap3:3chapchapduynhat:ws:theorem:27:08}
Let $\bf (H_i)$ hold for $i=1$ or $3$. Then for  all $p \in [1,\infty)$,
$$u(t) \to \varphi \quad \mbox{~~in~~} L^p(\Omega) \mbox{~~as~~} t \to \infty,$$
where $\varphi$ is given in Theorem \ref{uniqueness}.
\end{corollary}

\subsection*{Acknowledgements}
The author would like to thank Prof. Danielle Hilhorst for many helpful discussions.

\end{document}